\documentclass[reqno,12pt]{amsart}

\usepackage{amsthm,amsmath,amsfonts,amssymb}
\usepackage{hyperref}
\usepackage{color}
\usepackage[pagewise]{lineno}

\newtheorem{theorem}{Theorem}[section]
\newtheorem{definition}[theorem]{Definition}

\newtheorem{corollary}[theorem]{Corollary}
\newtheorem{remark}[theorem]{Remark}
\newtheorem{example}[theorem]{Example}
\renewcommand{\proof}{{\noindent \bf Proof:\ }}
\numberwithin{equation}{section}

\title[A Cayley-Hamilton-Ziebur Theorem for $X'=AX+F(t,X)$]{An extended version of the Cayley-Hamilton-Ziebur Theorem}	

\author[F. D. M. Bezerra]{Flank D. M. Bezerra$^*$}\thanks{$^*$Research partially supported by
CNPq \# 303039/2021-3, Brazil}
\address[F. D. M. Bezerra]{Departamento de Matem\'atica, Universidade Federal da Para\'iba, 58051-900 Jo\~ao Pessoa PB, Brazil.}
\email{flank@mat.ufpb.br}

\author[L. A. Santos]{Lucas A. Santos}
\address[L. A. Santos]{Instituto Federal da Para\'iba, 58051-900 Jo\~{a}o Pessoa PB, Brazil.}
\email{lucas92mat@gmail.com}

\begin{document}

\begin{abstract}
The aim of this paper is to prove a Cayley-Hamilton-Ziebur Theorem for non-autonomous semilinear matrix differential equations. Moreover we show the applicability of  results like these to ODE theory.

\end{abstract}

\maketitle

\tableofcontents

\section{Introduction}

Non-autonomous semilinear ordinary differential equations are important in many research areas, not just Mathematics. A significant part of the results related to the applicability of Mathematical in our daily lives are associated with the study of first-order semilinear systems with constant coefficients of the type
\begin{equation}\label{Equa01}
\begin{cases}
X'(t)=AX(t)+F(t,X(t)),\ t>\tau,\\ 
X(\tau)=X_0,\ \tau\in\mathbb{R},
\end{cases}
\end{equation}
where $A\in\mathcal{M}(n;\mathbb{C})$,  $X_0\in \mathcal{M}(n\times 1;\mathbb{C})$,  $X\in C^1([\tau,T);\mathcal{M}(n\times 1;\mathbb{C}))$, and\linebreak $F:\mathbb{R}\times \mathcal{M}(n\times 1;\mathbb{C})\to  \mathcal{M}(n\times 1;\mathbb{C})$  is a non-linear map with suitable conditions of growth and regularity in the sense of the classic theorems of existence and uniqueness of solution and continuous dependence on the initial data, (see for instance \cite{HJK}). Here, $\mathcal{M}(k\times \ell;\mathbb{C})$  denotes the set of matrices with  $k$ lines and $\ell$ columns with entries in  $\mathbb{C}$; $\mathcal{M}(k;\mathbb{C})$ denotes the set of matrices with $k$ lines and $k$ columns with entries in   $\mathbb{C}$; $tr(M)$ denotes the trace of  $M\in \mathcal{M}(k;\mathbb{C})$ and $det(M)$ denotes the determinant of $M\in \mathcal{M}(k;\mathbb{C})$; and $M^i=M\cdot M\cdot\ldots\cdot M$ denotes the product of $i$ copies of $M$.

An important result related to the problem \eqref{Equa01} is the Cayley-Hamilton-Ziebur theorem that establishes the relationship between the solution vector of \eqref{Equa01} and its components.
Let us recall what the Cayley-Hamilton-Ziebur theorem precisely says. See \cite{Ziebur} for more details.

\begin{theorem}{\textbf{(Cayley-Hamilton-Ziebur)}}\label{Theorem_0}
A component $x_j$, for $j=1,\dots,n$, of the vector solution $X(t)$ for $X'(t)=AX(t)$ is a solution of the $n$th order linear homogeneous constant-coefficient differential equation whose characteristic equation is $\det(A-\lambda I)=0$.
\end{theorem}

In other words, if 
\[
p_{A}(\lambda)=\lambda^n+a_{n-1}\lambda^{n-1}+\cdots+a_1\lambda+a_0
\]
denotes the characteristic polynomial of $A$, then Theorem \ref{Theorem_0} states that the vector solution $X=X(t)$ for $X'(t)=AX(t)$ satisfies

\begin{equation}\label{Eq0}
X^{(n)}+a_{n-1}X^{(n-1)}+\cdots+a_1X'+a_0X=0.
\end{equation}
Specifically, $X(t)$ is a vector linear combination of atoms constructed from the roots of the characteristic equation $\det(A-\lambda I) = 0$.

In particular, if $n=2$, then the components $x_1$ and $x_2$ of the solution application $X(t)=\begin{bmatrix} x_1\\ x_2\end{bmatrix}(t)$ of the problem $X'(t)=AX(t)$ is a solution of the linear homogeneous differential equation with constant coefficients of order $2$
\[
x''(t)-tr(A) x'(t)+det(A)x(t)=0.\\
\]
It is important recall that 
\[
det(A)=\dfrac{1}{2}[tr(A)^2-tr(A^2)].
\]

Analogously, if $n=3$, then the component $x_j$ (for $j=1,2,3$) of the solution application $X(t)=\begin{bmatrix} x_1\\ x_2\\ x_3 \end{bmatrix}(t)$ of the problem $X'(t)=AX(t)$ is a solution of the linear homogeneous differential equation with constant coefficients of order $3$
\[
x'''(t)-tr(A) x''(t)+ \dfrac{1}{2}[tr(A)^2-tr(A^2)]x'(t)-det(A)x(t)=0.
\]
It is important recall that 
\[
det(A)=\dfrac{1}{6}[tr(A)^3-3tr(A)tr(A^2)+2tr(A^3)].
\]

Roughly speaking, Cayley-Hamilton-Ziebur theorem provides an $n$th order differential equation whose the components of the solution of \eqref{Equa01} are solutions when the nonlinerity $F$ is the identically null application. Our main result in this paper provides a version of the Cayley-Hamilton-Ziebur Theorem for 
\[
X'=AX+F(t,X),
\]
where  $F: \mathbb{R}\times\mathcal{M}(n\times 1;\mathbb{C})\to  \mathcal{M}(n\times 1;\mathbb{C})$ is a non-linear application sufficiently smooth with suitable growth and regularity conditions. To the best of our knowledge, it is still not possible to find in the literature a discussion in this direction about non-autonomous semilinear ODE's.

\medskip

This paper is organized as follows.  In  Section \ref{Sec02}, we prove the main result of this paper. We also present some immediate consequences with some particular $F$. In  Section \ref{Sec03}, we present some applications.


\section{Main result}\label{Sec02}
Let us first introduce a definition of a class of polinomials associated with the characteristic polinomial of a matrix $A\in\mathcal{M}(n;\mathbb{C})$. These polinomials will play an important role in the main result.
\begin{definition}
Let  $A\in\mathcal{M}(n;\mathbb{C})$ with characteristic polynomial 
$$
p_{A}(\lambda)=\lambda^n+a_{n-1}\lambda^{n-1}+\cdots+a_1\lambda+a_0.
$$
We define, for $0\leqslant j\leqslant n$, the polynomial
$$
p_{j,A}(\lambda)=\lambda^j+a_{n-1}\lambda^{j-1}+\cdots+a_{n-j+1}\lambda+a_{n-j}.
$$
Note that $p_{n,A}(\lambda)=p_{A}(\lambda)$.
\end{definition}

The main result of this paper present a non-autonomous semilinear version of the Cayley-Hamilton-Ziebur Theorem, it is the following theorem.

\begin{theorem}\label{Theorem_1}
Let  $A\in\mathcal{M}(n;\mathbb{C})$,  $X_0=[x_{0j}]\in \mathcal{M}(n\times 1;\mathbb{C})$ and let\linebreak $F: \mathbb{R}\times\mathcal{M}(n\times 1;\mathbb{C})\to  \mathcal{M}(n\times 1;\mathbb{C})$ be a differentiable continuously map.  Let $X=X(t)\in C^1(\mathbb{R};\mathcal{M}(n;\mathbb{C}))$ be the unique solution of the problem  \eqref{Equa01}, if 
\[
p_{A}(\lambda)=\lambda^n+a_{n-1}\lambda^{n-1}+\cdots+a_1\lambda+a_0
\]
denotes the characteristic polynomial of $A$, then
\begin{equation}\label{Eqq}
\begin{cases}
\displaystyle X^{(n)}+a_{n-1}X^{(n-1)}+\cdots+a_1X'+a_0X-\sum_{j=0}^{n-2}p_{j,A}(A)\partial_t^{(n-j-1)}F(t,X)=p_{n-1,A}(A)F(t,X),\\
X(\tau)=X_0,\ \tau\in\mathbb{R},
\end{cases}
\end{equation}
Here, $a_{n}=1$, $X^{(0)}=X$ and $\partial_t^{(0)}F(t,X)=F(t,X)$.
\end{theorem}

This theorem is a sharpened version of Cayley-Hamilton-Ziebur Theorem, since it yields information about the nonlinearity $F$. Namely, the last term on the left-hand side and the term on the right-hand side of the first equation of \eqref{Eqq} provide the action of the nonlinearity $F$ comparing to \eqref{Eq0}.

The importance of this result can be seen from different aspects. From the point of view of the theory of non-autonomous semilinear differential equations, the formulation of the problem \eqref{Equa01} as in \eqref{Eqq} provide a new look at the problem and lead us to a different method to discuss the solvability and regularity of the equation in question. Theorem \ref{Theorem_1} can also be used to continue our previous analysis in \cite{BS} on evolution equation of third-order  in time, considering semilinear fractional approximations of these equations. To be more precise, let \eqref{Equa01} be an initial problem, a typical problem to be considered is the following spectral counterpart of \eqref{Equa01}

\begin{equation}\label{Equa01Spectral-f}
\begin{cases}
X_f'(t)=f(A)X(t)+F(t,X_f(t)),\ t>\tau,\\ 
X_f(\tau)=X_{f,0},\ \tau\in\mathbb{R},
\end{cases}
\end{equation}
where $f(A)$ is a matrix defined by spectral theory under suitable spectral conditions on $A$; namely, we can have $f(z)=z^\alpha$, $\alpha\in (0,1)$, $f(z)=\log z$, $f(z)=\sin(z)$, $f(z)=\cos(z)$,  etc. In this case, one may ask what is the scalar equation associated with \eqref{Equa01Spectral-f}? This type of questioning is associated with solvability and regularity of the equation \eqref{Equa01Spectral-f}, as well as,  with a comparison with the initial problem \eqref{Equa01}, see e.g. \cite{BCCN}, \cite{BS}, \cite{BBN}, \cite{MNS}.

Now we will prove the Theorem \ref{Theorem_1}.

\proof
Let
\[
r^n+a_{n-1}r^{n-1}+\cdots+a_1r+a_0=0
\]
be the characteristic equation of $A$. From the Cayley-Hamilton Theorem, we have
\begin{equation}\label{Eqcharac}
A^n+a_{n-1}A^{n-1}+\cdots+a_1A+a_0I=0.
\end{equation}
Right-multiplying \eqref{Eqcharac} by $X=X(t)$, we obtain
\[
A^nX+a_{n-1}A^{n-1}X+\cdots+a_1AX+a_0X=0.
\]
Now, differentiate $X'=AX+F(t,X)$ with respect to $t$ and see that
\[
X''=AX'+\partial_tF(t,X(t))
\]
that is
\[
X''=A^2X+AF(t,X)+\partial_tF(t,X(t)).
\]
Now, differentiate $X'=AX+F(t,X)$ two times with respect to $t$ and see that 
\[
X'''=AX''+\partial_t^2F(t,X(t)) 
\]
that is
\[
X'''=A^3X+A^2F(t,X)+A\partial_tF(t,X(t))+\partial_t^2F(t,X(t)).
\]
Now, differentiate $X'=AX+F(t,X)$ three times with respect to $t$ and see that
\[
X^{(4)}=AX'''+\partial_t^3F(t,X(t))
\]
that is
\[
X^{(4)}=A^4X+A^3F(t,X)+A^2\partial_tF(t,X(t))+A\partial_t^2F(t,X(t))+\partial_t^3F(t,X(t)).
\]
We continue in this fashion to obtain 
\[
X^{(k)}=A^kX+\displaystyle\sum_{j=0}^{k-1}A^j\partial_t^{(k-j-1)}F(t,X(t))
\]
for any $k\in\{1,\ldots,n\}$. Multiplying each $X^{(k)}$ by $a_{k}$, for $k\in\{1,\ldots,n\}$, and adding, we obtain
\begin{equation}\label{MainEqqq}
X^{(n)}+a_{n-1}X^{(n-1)}+\cdots+a_1X'+a_0X=\sum_{j=0}^{n-1}p_{j,A}(A)\partial_t^{(n-j-1)}F(t,X(t)). 
\end{equation}

Now note that
\[
\sum_{j=0}^{n-1}p_{j,A}(A)\partial_t^{(n-j-1)}F(t,X(t))=\sum_{j=0}^{n-2}p_{j,A}(A)\partial_t^{(n-j-1)}F(t,X(t))+p_{n-1,A}(A)F(t,X(t))
\]
and all the terms of $\sum_{j=0}^{n-2}p_{j,A}(A)\partial_t^{(n-j-1)}F(t,X(t))$ have at least one of the derivatives $X', X'', \ldots, X^{(n)}$ which leads us to rewrite \eqref{MainEqqq} with the term $\sum_{j=0}^{n-2}p_{j,A}(A)\partial_t^{(n-j-1)}F(t,X(t))$ on left side of \eqref{MainEqqq}, and therefore
\[
X^{(n)}+a_{n-1}X^{(n-1)}+\cdots+a_1X'+a_0X-\sum_{j=0}^{n-2}p_{j,A}(A)\partial_t^{(n-j-1)}F(t,X(t))=p_{n-1,A}(A)F(t,X(t)).\ \qed
\]
\qed

\begin{remark}
Note that the coefficients of the characteristic equation
\[
r^n+a_{n-1}r^{n-1}+\cdots+a_1r+a_0=0
\]
can be obtained from
\[
r^n+a_{n-1}r^{n-1}+\cdots+a_1r+a_0=\displaystyle\sum_{k=0}^nr^{n-k}(-1)^ktr(\Lambda^kA),
\]
where
\[
tr(\Lambda^kA)=\dfrac{1}{k!}det\begin{bmatrix}
tr(A) & k-1 & 0 & \cdots & 0\\
tr(A^2) & tr(A) &  k-2 &  \cdots & 0\\
\vdots & \vdots & \vdots & \ddots & \vdots \\ 
tr(A^{k-1}) & tr(A^{k-2}) &  tr(A^{k-3}) &  \cdots &  1\\
tr(A^k) & tr(A^{k-1}) &  tr(A^{k-2}) &  \cdots & tr(A)
\end{bmatrix}.
\]
\end{remark}

\begin{remark}
In particular, for $F\circ X$, it is convenient recall that the $k$th derivative of the composition $F\circ X$ at $t$ is
\[
\partial_t^{(k)}(F(X(t))=\displaystyle\sum\dfrac{k!}{i_1!i_2!\cdots i_k!}F^{(i_1+i_2+\cdots+i_k)}(X(t))\prod_{j=1}^k\Big(\dfrac{X^{(j)}(t)}{j!}\Big)^{i_j}
\]
where the sum is over all $k-$tuples $(i_1,i_2,\ldots,i_k)$ of nonnegative integers such that\linebreak $1\cdot i_1+2\cdot i_2+\cdots+k\cdot i_k=k$, this result is known in the literature as \textbf{Fa\`a di Bruno's Formula}, see e.g. \cite{Roman}.
\end{remark}

\begin{corollary}\label{Corol1}
Let $F: \mathbb{R}\times \mathcal{M}(2\times 1;\mathbb{C})\to  \mathcal{M}(2\times 1;\mathbb{C})$ be a continuously differentiable function given by
\[
F(t,X)=\begin{bmatrix} 
0\\ f(t,x_1)
\end{bmatrix}
\]
for any $X=\begin{bmatrix} 
x_1\\ x_2
\end{bmatrix}$, where $f:\mathbb{R}^2\to\mathbb{R}$ is a function with suitable conditions of growth and regularity. Then the solution application $X=[x_j]$ of the semilinear Cauchy problem \eqref{Equa01} satisfies 
\[
\begin{cases}
X''(t)-tr(A)X'(t)+det(A)X(t)-\partial_tF(t,X(t))=(-tr(A)I+ A)F(t,X(t)),\ t>\tau,\\
X(\tau)=X_0=\begin{bmatrix}
x_{01}\\x_{02}
\end{bmatrix}, \tau\in\mathbb{R}.
\end{cases}
\]
More precisely, $x_1$ is the unique solution of the non-autonomous semilinear Cauchy problem
\begin{equation}\label{EqEntrada21}
\begin{cases}
x''(t)-tr(A)x'(t)+det(A)x(t)=\alpha_{12}f(t,x(t)),\ t>\tau,\\
x(\tau)=x_{01},\ x'(\tau)=\alpha_{11}x_{01}+\alpha_{12}x_{02},\ \tau\in\mathbb{R},
\end{cases}
\end{equation}
where $A=[\alpha_{ij}]\in\mathcal{M}(n;\mathbb{C})$. 
\end{corollary}

\proof
The result follows immediately from Theorem \ref{Theorem_1} for $n=2$, see that 
\begin{eqnarray*}
p_{0,A}(A)&=&I\\
p_{1,A}(A)&=&-tr(A) I+A 
\end{eqnarray*}
and 
\begin{eqnarray*}
\partial_tF(t,X(t))&=&\begin{bmatrix} 0\\ \partial_tf(t,x_1)\end{bmatrix}\\
(-tr(A)I+ A)F(t,X(t))&=& \begin{bmatrix} -\alpha_{22}&\alpha_{12}\\ \alpha_{21}&-\alpha_{11}\end{bmatrix}\begin{bmatrix} 0\\ f(t,x_1)\end{bmatrix}=\begin{bmatrix} \alpha_{21}f(t,x_1)\\ -\alpha_{11}f(t,x_1)\end{bmatrix} \qed
\end{eqnarray*}
\qed

\begin{corollary}\label{Co3}
Let $F: \mathbb{R}\times\mathcal{M}(3\times 1;\mathbb{C})\to  \mathcal{M}(3\times 1;\mathbb{C})$ be a continuously differentiable function given by
\begin{equation}\label{asdf4asasa6}
F(t,X)=\begin{bmatrix} 
0\\ 0\\ f(t,x_1)
\end{bmatrix}
\end{equation}
for any $X=\begin{bmatrix} 
x_1\\ x_2\\ x_3
\end{bmatrix}$, where  $f:\mathbb{R}^2\to\mathbb{R}$ is a function with suitable conditions of growth and regularity. Then the solution application $X=[x_j]$ of the  Cauchy problem \eqref{Equa01} satisfies 
\[
\begin{cases}
X'''(t)-tr(A)X''(t)+ \dfrac{1}{2}[tr(A)^2-tr(A^2)]X'(t)-det(A)X(t)-\partial_t^2F(t,X(t))\\
-(A+\dfrac{1}{2}[tr(A)^2-tr(A^2)]I)\partial_tF(t,X(t))(t)=(A^2-tr(A)A+\dfrac{1}{2}[tr(A)^2-tr(A^2)]I)F(t,X(t)),\\
X(\tau)=X_0=\begin{bmatrix}
x_{01}\\x_{02}\\x_{03}
\end{bmatrix},\ \tau\in\mathbb{R}.
\end{cases}
\]
More precisely, $x_1$ is the unique  solution of the non-autonomous semilinear Cauchy problem
\begin{equation}\label{EqEntrada21}
\begin{cases}
x'''(t)-tr(A)x''(t)+ \dfrac{1}{2}[tr(A)^2-tr(A^2)]x'(t)-det(A)x(t)-\alpha_{13}\partial_tf(t,x)=\\
=(\alpha_{12}\alpha_{23}-\alpha_{22}\alpha_{13})f(t,x),\ t>\tau,\\
x(\tau)=x_{01},\ x'(\tau)=\alpha_{11}x_{01}+\alpha_{12}x_{02}+\alpha_{13}x_{03},\\ x''(\tau)=(\alpha_{11}^2+\alpha_{12}\alpha_{21}+\alpha_{13}\alpha_{31})x_{01}+(\alpha_{11}\alpha_{22}+\alpha_{12}\alpha_{22}+\alpha_{13}\alpha_{32})x_{02}+\\
+(\alpha_{11}\alpha_{13}+\alpha_{12}\alpha_{23}+\alpha_{13}\alpha_{33})x_{03} +\alpha_{13}f(0,x_{01}),\ \tau\in\mathbb{R}.
\end{cases}
\end{equation}
\end{corollary}

\proof
The result follows immediately from Theorem \ref{Theorem_1} for $n=3$, since 
\[
p_{0,A}(A)=I,\ \ p_{1,A}(A)= A+\dfrac{1}{2}[tr(A)^2-tr(A^2)] I,\ \ p_{2,A}(A)= A^2-tr(A)A+\dfrac{1}{2}[tr(A)^2-tr(A^2)]I,
\]
and 
\begin{eqnarray*}
\partial_t^2F(t,X(t))(t)&=&\begin{bmatrix} 0\\ 0\\ \partial_t^2f(t,x_1)\end{bmatrix}\\
(A+\dfrac{1}{2}[tr(A)^2-tr(A^2)]I)\partial_tF(t,X(t))(t) &=& \begin{bmatrix} \alpha_{13}\partial_t f(t,x_1)\\ e_{21} \\ e_{31} \end{bmatrix}\\
(A^2-tr(A)A+\dfrac{1}{2}[tr(A)^2-tr(A^2)]I)F(t,X(t))(t) &=& \begin{bmatrix} (\alpha_{12}\alpha_{23}-\alpha_{22}\alpha_{13})f(t,x_1)\\ g_{21} \\ g_{31} \end{bmatrix}.
\end{eqnarray*}
Note that the explicity definition of the terms $e_{21},\ e_{31},\ g_{21}$ and $g_{31}$ does not matter, as we are only dealing with the first entry of the matrix equations above. \qed
\qed

\begin{remark}
We can consider the problem  \eqref{Equa01} with $A=A(t)\in\mathcal{M}(n;\mathbb{C})$, $t\in I$, where $I\subset \mathbb{R}$  is an interval  and $F: \mathbb{R}\times\mathcal{M}(n\times 1;\mathbb{C})\to  \mathcal{M}(n\times 1;\mathbb{C})$ it is a non-linear application with suitable growth and regularity conditions; that is, consider the singularly non-autonomous semilinear problem
\[
\begin{cases}
X'(t)+A(t)X(t)=F(t,X(t)),\ t>\tau,\\
X(\tau)=X_0,\ \tau\in\mathbb{R}.
\end{cases}
\]
To the best of our knowledge, it is still not possible to find in the literature a discussion in this direction about  singularly non-autonomous semilinear ODE's. Of course, if $A(t)=A+B(t)$ and $A,B(t)\in\mathcal{M}(n;\mathbb{C})$ are non-zero matrices and the application $[0,\infty)\ni t\mapsto B(t)\in \mathcal{M}(n;\mathbb{C})$ is continuously differentiable, then we can rewrite the singularly non-autonomous semilinear equation 
\[
X'(t)+A(t)X(t)=F(t,X(t))
\]
as the following  non-autonomous semilinear equation
\[
X'(t)+AX(t)=G(t,X(t))
\]
where $G(t,X(t))=F(t,X(t))-B(t)$ and our results apply in the context of local solubility of the equations.
\end{remark}

\begin{remark}
It is well-know that coefficients of the characteristic polynomial  can also be expressed directly in terms of the eigenvalues of $A$, as shown in \cite{BPB}, therefore the problem \eqref{Eqq} can be present in terms of the eigenvalues of $A$.
\end{remark}

\section{Applications}\label{Sec03}

In section we show some applications of our previous results in the cases $n=2$ and $n=3$. For a nonsingular matrix $A\in\mathcal{M}(k;\mathbb{C})$ and $\alpha\in(0,1)$ we define 
\[
A^\alpha=e^{\alpha\log A}.
\] 
Here the logarithm is the principal matrix logarithm, the matrix function built on the principal scalar logarithm, and so the eigenvalues of $\log A$ lie in  $\{z\in\mathbb{C};-\pi<Im z\leqslant \pi\}$. For positive definite matrices and $\alpha\in(0,1)$, an  integral expression for $A^\alpha$ valid   is 
\begin{equation}\label{eq:to-compute-fractional-powers}
A^{\alpha}=\dfrac{\sin \alpha \pi}{\pi}\int_0^\infty\lambda^{\alpha-1}A(\lambda I+A)^{-1}d\lambda.
\end{equation}

\begin{example}
For $2\times 2$ real matrices of the form
\[
A=
\begin{bmatrix}
a & b\\
c & a
\end{bmatrix}
\]
with $bc < 0$ we have an explicit formula for $A^\alpha$ with $\alpha\in(0,1)$. It is easy to see that $A$ has eigenvalues $\lambda_{\pm}=a\pm id$, where $d=\sqrt{-bc}$. Let $\theta=\mbox{arg}(\lambda_+) \in (0,\pi)$ and $r = |\lambda_+|$.   It can be shown that
\begin{equation}\label{FractPow2}
A^\alpha=
\dfrac{r^\alpha}{d}
\begin{bmatrix}
d\cos(\alpha\theta) & b\sin(\alpha\theta)\\
c\sin(\alpha\theta) & d\cos(\alpha\theta)
\end{bmatrix}.
\end{equation}

A typical problem of approximation theory in differential equations and its fractional approximations can be seen in simple harmonic motions; namely, consider a mass $m>0$ suspended from a spring attached to a rigid support. Gravity is pulling the mass downward and the restoring force of the spring is pulling the mass upward, when these two forces are equal, the mass is said to be at the equilibrium position. If the mass is displaced from equilibrium, it oscillates up and down. Suppose also that the mass is subject to time-dependent external forces. This behavior can be modeled by a non-autonomous semilinear second-order constant-coefficient differential equation. 

Let $x=x(t)$ denote the displacement of the mass from equilibrium, it is customary to adopt the convention that down is positive. Thus, a positive displacement indicates the mass is below the equilibrium point, whereas a negative displacement indicates the mass is above equilibrium. Thanks to  Hooke’s law and Newton’s second law we have
\begin{equation}\label{EqASD000}
x''(t)+\dfrac{k}{m} x(t)=f(t,x(t)),
\end{equation}
where  $k>0$ is a constant factor characteristic of the spring  and $f:\mathbb{R}^2\to\mathbb{R}$ is a function with suitable conditions of growth and regularity. Note that \eqref{EqASD000}  can be rewrite as a semilinear matrix differential equation
\begin{equation}\label{EqFracPoASDjhjh}
X'(t)=\varLambda_{(\omega)} X(t)+F(t,X(t)),
\end{equation}
where $\omega=\sqrt{\dfrac{k}{m}}$, 
\[
X(t)=\begin{bmatrix} x(t) \\ x'(t) \end{bmatrix},\quad \varLambda_{(\omega)}=
\begin{bmatrix}
0 & 1\\
-\omega^2 & 0
\end{bmatrix},\quad F(t,X(t))=\begin{bmatrix} 0 \\ f(t,x(t)) \end{bmatrix}.
\]

Using \eqref{FractPow2} we have
\[
\varLambda_{(\omega)}^\alpha=
\omega^{\alpha-1}
\begin{bmatrix}
 \omega\cos\Big(\dfrac{\alpha\pi}{2}\Big) & \sin\Big(\dfrac{\alpha\pi}{2}\Big)\\ 
-\omega^2\sin\Big(\dfrac{\alpha\pi}{2}\Big) & \omega\cos\Big(\dfrac{\alpha\pi}{2}\Big).
\end{bmatrix}
\]
we can to present a class of fractional approximations of \eqref{EqFracPoASDjhjh} given by 
\begin{equation}\label{EqFracPoASD}
X_\alpha'(t)=\varLambda_{(\omega)}^\alpha X_\alpha(t)+F(t,X_\alpha(t)),\quad \alpha\in(0,1),
\end{equation}
and from the point of view of  semilinear scalar differential equations, thanks to  Corollary \ref{Corol1} (see \eqref{EqEntrada21}) and \eqref{FractPow2}, we have the following fractional approximations of \eqref{EqASD000}
\begin{equation}\label{EqASDfrac}
x_\alpha''(t)+2\omega^\alpha\cos\Big(\frac{\alpha\pi}{2}\Big)x'_\alpha(t)+\omega^{\alpha+1} x_\alpha(t)=\omega^{\alpha-1}\sin\Big(\frac{\alpha\pi}{2}\Big) f(t,x_\alpha(t)),
\end{equation}
for $\alpha\in(0,1)$. The presence of the term $2\omega^\alpha\cos(\frac{\alpha\pi}{2})x'_\alpha(t)$ in \eqref{EqASDfrac} with $2\omega^\alpha\cos(\frac{\alpha\pi}{2})>0$ for any  $\alpha\in(0,1)$ allows us to conclude that the energy of system \eqref{EqASD000} is dissipated for long time. 
\end{example}

\begin{example}
Third-order ordinary differential equations arise from a variety of different areas of applied mathematics and physics, e.g.,  in the deflection of a curved beam having a constant or varying cross-section, a three-layer beam, electromagnetic waves or gravity driven flows, see e.g. \cite{ZZSM}. To fix our attention, consider the non-autonomous semilinear ordinary differential equation
\begin{equation}\label{EqASD}
x'''(t)+\beta x(t)=f(t,x(t)),
\end{equation}
where $\beta>0$ and $f:\mathbb{R}\to\mathbb{R}$ is a function with suitable conditions of growth and regularity. Note that \eqref{EqASD}  can be rewrite as a semilinear matrix differential equation
\begin{equation}\label{EqFracPoASD}
X'(t)=\varLambda_{(\beta)} X(t)+F(t,X(t)),
\end{equation}
where
\[
X(t)=\begin{bmatrix} x(t) \\ x'(t)\\ x''(t) \end{bmatrix},\quad \varLambda_{(\beta)}=
\begin{bmatrix}
0 & 1 & 0\\
0 & 0 & 1\\
-\beta & 0 & 0
\end{bmatrix},\quad F(t,X(t))=\begin{bmatrix} 0 \\ 0\\ f(t,x(t)) \end{bmatrix}.
\]

Using \eqref{eq:to-compute-fractional-powers} we have
\[
\varLambda_{(\beta)}^\alpha=
\left[\begin{matrix}
-k_{\alpha,0}\beta^{\frac{\alpha}{3}}&k_{\alpha,2}\beta^{\frac{\alpha-2}{3}}&-k_{\alpha,1}\beta^{\frac{\alpha+1}{3}}\\
k_{\alpha,1}\beta^{\frac{\alpha+4}{3}}&-k_{\alpha,0}\beta^{\frac{\alpha}{3}}&k_{\alpha,2}\beta^{\frac{\alpha-2}{3}}\\
-k_{\alpha,2}\beta^{\frac{\alpha+1}{3}}&k_{\alpha,1}\beta^{\frac{\alpha+4}{3}}&-k_{\alpha,0}\beta^{\frac{\alpha}{3}}
\end{matrix}\right]
\]
where
\[
k_{\alpha,j} = \frac13\left(2\cos{\tfrac{2\pi (\alpha+j)}{3}}+1\right),\ \ \text{for\ } j\in\{0,1,2\}.
\]
Moreover, these coefficients satisfy the following properties
\[
\det \left[\begin{matrix}
-k_{\alpha,0}&k_{\alpha,2}&-k_{\alpha,1}\\
k_{\alpha,1}&-k_{\alpha,0}&k_{\alpha,2}\\
-k_{\alpha,2}&k_{\alpha,1}&-k_{\alpha,0}
\end{matrix}\right]= 1,
\]
\[
k_{\alpha,0}+k_{\alpha,1}+k_{\alpha,2} = 1,
\]
and
\[
\begin{cases}
k_{\alpha,0}^2-k_{\alpha,1}k_{\alpha,2} = k_{\alpha,0}\\
k_{\alpha,1}^2-k_{\alpha,0}k_{\alpha,2} = k_{\alpha,1}\\
k_{\alpha,2}^2-k_{\alpha,0}k_{\alpha,1} = k_{\alpha,2}.
\end{cases}
\]

We can to present a class of fractional approximations of \eqref{EqFracPoASD} given by 
\begin{equation}\label{EqFracPoASD}
X_\alpha'(t)=\varLambda_{(\beta)}^\alpha X_\alpha(t)+F(t,X_\alpha(t)),\quad \alpha\in(0,1),
\end{equation}
and from the point of view of  semilinear scalar differential equations, thanks to  Corollary \ref{Co3} (see \eqref{EqEntrada21}) and \eqref{FractPow2}, we have the following fractional approximations of \eqref{EqASD} 
\begin{equation}\label{EqASD}
x_\alpha'''(t)-\Big[2\cos\Big({\dfrac{2\pi \alpha}{3}}\Big)+1\Big]\beta^{\frac{\alpha}{3}} x_\alpha''(t)-\Big[2\cos\Big({\dfrac{2\pi \alpha}{3}}\Big)+1\Big]\beta^{\frac{2\alpha}{3}} x'_\alpha(t)+\beta^\alpha x_\alpha(t)=f(t,x_1(t)),
\end{equation} 
for $\alpha\in(0,1)$.
\end{example}

 \end{document}